\documentclass[a4paper,UKenglish,cleveref, autoref, thm-restate]{lipics-v2021}

\pdfoutput=1 
\hideLIPIcs  

\title{Edge Partitions of Complete Geometric Graphs (Part 2)} 

\author{Oswin Aichholzer}{Institute of Software Technology, Graz
University of Technology,
Austria}{oaich@ist.tugraz.at}{https://orcid.org/0000-0002-2364-0583}{Partially
supported by the Austrian Science Fund (FWF): W1230 and the European
Union H2020-MSCA-RISE project 73499 - CONNECT.}

\author{Johannes Obenaus}{Department of Computer Science, Free University of Berlin,
Germany}{johannes.obenaus@fu-berlin.de}{https://orcid.org/0000-0002-0179-125X}{Supported by ERC StG 757609.}

\author{Joachim Orthaber}{Institute of Software Technology, Graz University of Technology, Austria}{joachim.orthaber@student.tugraz.at}{https://orcid.org/0000-0002-9982-0070}{}

\author{Rosna Paul}{Institute of Software Technology, Graz University of
Technology,
Austria}{ropaul@ist.tugraz.at}{https://orcid.org/0000-0002-2458-6427}{Supported
by the Austrian Science Fund (FWF): W1230.}

\author{Patrick Schnider}{Department of Mathematical Sciences, University
of Copenhagen, Denmark \and
\url{http://people.inf.ethz.ch/schnpatr}}{ps@math.ku.dk}{https://orcid.org/0000-0002-2172-9285}{Has
received funding from the European Research Council under the European
Unions Seventh Framework Programme ERC Grant agreement ERC StG 716424 -
CASe.}

\author{Raphael {Steiner}}{Departement of Computer Science, ETH Z\"{u}rich, Switzerland \and \url{
https://sites.google.com/view/raphael-mario-steiner/startseite} 
}{raphaelmario.steiner@inf.ethz.ch}{
https://orcid.org/0000-0002-4234-6136}{Supported by an ETH Zurich 
Postdoctoral Fellowship}

\author{Tim Taubner}{Departement of Computer Science, ETH Z\"{u}rich,
Switzerland}{tim.taubner@inf.ethz.ch}{https://orcid.org/0000-0001-5786-4756}{}

\author{Birgit Vogtenhuber}{Institute of Software Technology, Graz
University of Technology, Austria
}{bvogt@ist.tugraz.at}{https://orcid.org/0000-0002-7166-4467}{Partially
supported by Austrian Science Fund within the collaborative DACH project
\emph{Arrangements and Drawings} as FWF project \mbox{I 3340-N35}.}

\authorrunning{Aichholzer, Obenaus, Orthaber, Paul, Schnider, Steiner, Taubner, Vogtenhuber} 

\Copyright{Oswin Aichholzer, Johannes Obenaus, Joachim Orthaber, Rosna Paul, Patrick Schnider, Raphael Steiner, Tim Taubner, Birgit Vogtenhuber} 



\category{} 

\relatedversion{part 1: \url{https://arxiv.org/abs/2108.05159}} 

\acknowledgements{We thank the organizers and participants of the $5^{th}$ DACH Workshop on Arrangements and Drawings, that took place in March 2021 (online) and was funded by Deutsche Forschungsgemeinschaft (DFG), the Austrian Science Fund (FWF) and the Swiss National Science Foundation (SNSF).}

\nolinenumbers 

\EventEditors{John Q. Open and Joan R. Access}
\EventNoEds{2}
\EventLongTitle{42nd Conference on Very Important Topics (CVIT 2016)}
\EventShortTitle{CVIT 2016}
\EventAcronym{CVIT}
\EventYear{2016}
\EventDate{December 24--27, 2016}
\EventLocation{Little Whinging, United Kingdom}
\EventLogo{}
\SeriesVolume{42}
\ArticleNo{23}

\usepackage{graphicx}
\graphicspath{{./figures/}}
\usepackage{wrapfig}

\usepackage[draft]{todonotes}   

\usepackage{mathtools}
\usepackage{csquotes}

\newcommand{\E}{\mathbb{E}}
\newcommand{\crn}{\operatorname{cr}}

\usepackage{xspace}
\newcommand{\quasiplanar}{quasi-planar\xspace}
\newcommand{\quasiplane}{quasi-plane\xspace}

\begin{document}

\maketitle

\begin{abstract}
	Recently, the second and third author showed that complete geometric graphs on $2n$ vertices in general cannot be partitioned into $n$ plane spanning trees. Building up on this work, in this paper, we initiate the study of partitioning into beyond planar subgraphs, namely into $k$-planar and $k$-\quasiplanar subgraphs and obtain first bounds on the number of subgraphs required in this setting.
\end{abstract}

\section{Introduction}\label{sec:intro}
A geometric graph $G = G(P,E)$ is a drawing of a graph in the plane where the vertex set is drawn as a point set $P$ in general position 
(that is, no three points are collinear) and each edge of $E$ is drawn as a straight-line segment between its vertices.
A geometric graph $G$ is \emph{plane} if no two of its edges \emph{cross} (that is, share a point in their relative interior). Similarly, $G$ is \emph{$k$-plane} for some $k \geq 0$, if each edge crosses at most $k$ other edges and \emph{$k$-\quasiplane} if there are no $k$ pairwise crossing edges (or in other words, a largest \emph{crossing family} has size less than $k$). Further, we call $G$ \emph{convex}, if the underlying point set is in convex position.

A \emph{partition} (also called edge partition) of a graph $G$ is a set of edge-disjoint subgraphs of $G$ whose union is $G$.
A subgraph of (a connected graph) $G$ is \emph{spanning} if it is connected and its vertex set is the same as the one of $G$.

\subparagraph*{Related Work.} Recently, the second and third author \cite{obenaus_orthaber_2021} showed that there are complete geometric graphs that cannot be partitioned into plane spanning trees, while this is always possible for graphs in convex \cite{kainen1979} or regular wheel position \cite{aichholzer2017packing,trao2019edge}. For the related \emph{packing} problem where not all edges of the underlying graphs must be covered, Biniaz and Garc{\'{\i}}a \cite{garcia_packing} showed that $\lfloor n/3 \rfloor$ plane spanning trees can be packed in any complete geometric graph on $n$ vertices, which is the currently best lower bound.

Note that the problem of partitioning a geometric graph into plane subgraphs is equivalent to a classic edge coloring problem, where each edge should be assigned a color in such a way that no two edges of the same color cross (of course using as few colors as possible). This problem received considerable attention from a variety of perspectives (see for example~\cite{DBLP:journals/jct/PawlikKKLMTW14} and references therein) and is the topic of the \href{https://cgshop.ibr.cs.tu-bs.de/competition/cg-shop-2022/\#problem-description}{CG:SHOP challenge 2022} \cite{CG_challenge}.

\subparagraph*{Our Results.} In this paper, we study the broader setting of beyond planar partitions. To the best of our knowledge, this problem has not been studied before. In the setting of $k$-planar partitions we focus on the convex setting, showing the following bounds (which is tight for $k=1$):

\begin{restatable}{proposition}{propThreePlanarcase}\label{prop:3planarcase}
For a point set $P$ in convex position with $|P| = n \ge 5$, $K(P)$ can be partitioned into $\left\lceil\frac{n}{3}\right\rceil$ 1-planar subgraphs and $\left\lceil\frac{n}{3}\right\rceil$ subgraphs are required in every 1-planar partition.
\end{restatable}

\begin{restatable}{theorem}{thmGeneral}\label{thm:general}
For an $n$-point set $P$ in convex position and every $k \in \mathbb{N}$, $K(P)$ admits a partition into at most $\frac{n}{\sqrt{2k}}$ $k$-planar subgraphs. More precisely, for every $s > 2$, $K(P)$ admits a $\frac{(s-1)(s-2)}{2}$-planar partition into $\lceil\frac{n}{s}\rceil$ subgraphs.

Conversely, for every $k \in \mathbb{N}$, at least $\frac{n-1}{4.93\sqrt{k}}$ subgraphs are required in any $k$-planar partition of $K(P)$.
\end{restatable}

On the other hand, we consider partitions into $k$-\quasiplanar subgraphs for arbitrary point sets (in general position). We show that a partition into $3$-\quasiplanar spanning trees is possible for any $P$ with $|P|$ even. This is best possible, as $2$-\quasiplanar graphs are plane. We further present bounds on the partition of any $K(P)$ into $k$-\quasiplanar subgraphs for general $k$.

\begin{restatable}{lemma}{lemThreeQuasiplanarST}\label{lem:3_quasiplanar_st}
  Let $P$ be a point set of size $2n$, then the complete geometric graph $K(P)$ can be partitioned into $n$ 3-\quasiplanar spanning trees.
\end{restatable} 

\begin{restatable}{theorem}{thmMainQuasiplanarCF}\label{thm:main_quasiplanar_cf}
Let $P$ be a set of $n$ points in general position and denote the size of a largest crossing family  on $P$ by $m$. Also let $k \geq 3$ s.t. $k \leq m$ (otherwise one color is always sufficient). Then, at least $\lceil \frac{m}{k-1} \rceil$ colors are required and at most $\lceil \frac{m}{k-1} \rceil + \lceil \frac{n-2m}{k-1} \rceil$ colors are needed to partition the complete geometric graph $K(P)$ into $k$-\quasiplanar subgraphs.
\end{restatable}

\Cref{sec:k_planar_partitions} is dedicated to the $k$-planar setting, where we prove \Cref{prop:3planarcase} and \Cref{thm:general}, while \Cref{sec:k-quasi-planar} is concerned with $k$-\quasiplanar subgraphs, proving \Cref{lem:3_quasiplanar_st} and \Cref{thm:main_quasiplanar_cf}.

\section{Partitions into $\mathbf{k}$-planar subgraphs}\label{sec:k_planar_partitions}

For the proofs of Proposition~\ref{prop:3planarcase} and Theorem~\ref{thm:general} (in particular for the lower bounds) it will be necessary to understand how many edges a single color class\footnote{It is more convenient to associate edges with colors instead of speaking about subgraphs of a partition.}, or in other words, how many edges a $k$-planar subgraph of a convex geometric $K_n$, can maximally have. Once such bounds are established, we will be able to lower-bound the number of colors required in any $k$-planar partition of a convex geometric $K_n$ by considering the ``largest'' color class. 

The following results are the main ingredients towards the proofs of \Cref{prop:3planarcase} and \Cref{thm:general}.

\begin{proposition}\label{prop:smallvalues}
For every $k \in \{0,1,2,3,4\}$, every convex $k$-plane graph $G$ on $n \ge 2$ vertices has at most $\frac{k+4}{2}n-(k+3)$ edges.  
\end{proposition}
 For larger values of $k$, we have the following general bound. 
 
\begin{theorem}\label{thm:generalupperbound}
For every $k \ge 5$, every convex $k$-plane graph $G$ on $n$ vertices has at most $\sqrt{\frac{243}{40}k} \cdot n$ edges.
\end{theorem}

On the way to the proof of Theorem~\ref{thm:generalupperbound}, we also establish the following improvement of the well-known crossing lemma for geometric graphs whose vertices are in convex position:

\begin{lemma}[convex crossing lemma]\label{lemma:convexcrossing}
Let $G$ be a graph with $n$ vertices and $e$ edges such that $e \ge \frac{9}{2}n$. Then every straight-line embedding of $G$ into the plane in which the vertices of $G$ are placed in convex position has at least 
\[
\frac{20}{243}\frac{e^3}{n^2} \approx 0.0823 \frac{e^3}{n^2}
\] 
crossings. 
\end{lemma}

The rest of the section is structured as follows: First we deduce Proposition~\ref{prop:smallvalues} from a lemma proved by Pach and T\'{o}th in~\cite{pach1997crossings}. We will then use the bounds from Proposition~\ref{prop:smallvalues} to carry through an analogue of the well-known probabilistic argument for the crossing lemma to establish the convex crossing lemma. Finally, we use a double-counting argument for the number of crossings to prove the bound in Theorem~\ref{thm:generalupperbound}. 

Let us start with the following Lemma by Pach and T\'{o}th. To state the lemma properly, we need some terminology as follows: 

Given a (multi-)graph $G$ drawn in the plane without crossings, a \emph{face} $\Phi$ of $G$ is a maximal connected region in the complement of the drawing. Following the notation in~\cite{pach1997crossings}, the number $|\Phi|$ of \emph{sides} of $\Phi$ is defined as the number of edges on the boundary walk of $\Phi$, where edges whose both sides are on the boundary walk of $\Phi$ are counted with multiplicity $2$. Furthermore, by $t(\Phi)$ we denote the number of triangles in a triangulation of the inside of $\Phi$ (using vertices from $G$).
We also need the following notion: Given a (multi-)graph $M$ drawn in the plane, consider a crossing-free subgraph $M'$ of $M$ with a maximum number of edges. Then every edge $e \in E(M) \setminus E(M')$ clearly must cross at least one edge of $M'$. The closed portion between an endpoint of $e$ and the nearest crossing of $e$ with an edge of $M'$ is called a \emph{half-edge}. Note that every edge in $E(M) \setminus E(M')$ contributes exactly two half-edges.

\begin{lemma}[cf.~\cite{pach1997crossings}, Lemma~2.2]\label{lemma:technical}
Let $k \in \{0,1,2,3,4\}$ and let $M$ be a (multi-)graph drawn in the plane so that every edge crosses at most $k$ others. Let $M'$ be a planar subgraph of $M$ with a maximum number of edges. Let $\Phi$ be a face of $M'$ with $|\Phi| \ge 3$ sides in $M'$. Then there are at most $t(\Phi)k+|\Phi|-3$ half-edges contained in the closed interior of $\Phi$. 
\end{lemma}

With this lemma at hand, we can now easily deduce Proposition~\ref{prop:smallvalues}. 
\begin{proof}[Proof of Proposition~\ref{prop:smallvalues}]
Let $k \in \{0,1,2,3,4\}$ and let $G$ be a given $n$-vertex graph embedded geometrically into the plane such that every edge crosses with at most $k$ other edges. Let $G'$ be a planar subgraph of $G$ with a maximum number of edges. W.l.o.g. we may assume that $G'$ contains all the $n$ vertices and all $n$ boundary edges (that is, edges between consecutive vertices in the cyclical ordering of the vertices, since such edges can always be added without introducing any additional crossings).  Then $G'$ is a straight-line embedding of an outerplanar $n$-vertex graph. For each face $\Phi$ of $G'$, pick and fix a triangulation of this face using $t(\Phi)$ triangles. Now consider the union of all these triangulations. Clearly, this union forms a triangulation of a convex $n$-gon, and hence, has exactly $n-2$ triangles. It follows, denoting by $\mathcal{F}$ the set of all interior faces of $G'$: 
\[
\sum_{\Phi \in \mathcal{F}}{t(\Phi)}=n-2.
\]
Note that every half-edge by definition does not cross any edge of $G'$ and hence is fully contained in one of the interior faces of $G'$. Using Lemma~\ref{lemma:technical} we find that the number of half-edges is at most
\begin{align*}
\sum_{\Phi \in \mathcal{F}}{\left(t(\Phi)k+|\Phi|-3\right)}&=
k\sum_{\Phi \in \mathcal{F}}{t(\Phi)}+\sum_{\Phi \in \mathcal{F}}{|\Phi|}-3|\mathcal{F}| \\
&=k(n-2)+(2e(G')-n)-3(f(G')-1),
\end{align*}
where we used the fact that the outer face of $G'$ contributes exactly $n$ edges which are counted only once in the sum $\sum_{\Phi \in \mathcal{F}}{|\Phi|}$.
Clearly, the number of half-edges equals $2(e(G)-e(G'))$, and hence we conclude:
\[
2e(G)-2e(G') \le k(n-2)+(2e(G')-n)-3f(G')+3.
\] 
Rearranging, applying Euler's formula $n-e(G')+f(G')=2$, and using the fact that $f(G') \le n-1$ now yields:
\begin{align*}
2e(G) &\le k(n-2)-n+3+4e(G')-3f(G') \\
&=k(n-2)-n+3+4(e(G')-f(G'))+f(G') \\
&\le k(n-2)-n+3+4(n-2)+(n-1)=(k+4)n-(2k+6).
\end{align*}

Dividing by two yields the claimed bound $e(G) \le \frac{k+4}{2}n-(k+3)$, as desired.
\end{proof}

Let us now give the proof of the ``Convex Crossing Lemma''. 

\begin{proof}[Proof of Lemma~\ref{lemma:convexcrossing}]
We consider the following process: We start with $G$, and repeatedly update a subgraph of $G$. As long as the current subgraph still contains crossings, pick an edge with the highest number of crossings from the current drawing and remove it. Repeat until we end up with a crossing-free subgraph. As long as the current subgraph has more than $4n-7$ edges, by Proposition~\ref{prop:smallvalues} applied for $k=4$, it follows that the current subgraph contains edges which cross with at least $5$ other edges, and hence, in the next step the edge we remove will remove at least $5$ crossings from $G$. Similarly, as long as the number of edges in the subgraph is strictly greater than $\frac{7}{2}n-6$, by Proposition~\ref{prop:smallvalues} the edges we remove will be such with at least $4$ crossings, if it is strictly greater than $3n-5$, then we remove at least $3$ crossings at each step, etc. 

Summing up all these different contributions of how many crossings are lost in the process by removing edges, we obtain the following lower bound on the number of crossings of $G$:
\begin{align*}
\crn(G) &\geq 5\left(e - (4n-7)\right) + 4\left((4n-7) - \left(\frac{7}{2}n-6\right)\right) + 3\left(\left(\frac{7}{2}n-6\right) - (3n-5)\right) + \\
&2\left((3n-5)-\left(\frac{5}{2}n-4\right)\right) + \left(\left(\frac{5}{2}n-4\right) - (2n-3)\right) = 5e - 15n+25.
\end{align*}

So, let $p \in [0,1]$ be a probability to be determined later and let $G_p$ be the induced random subgraph of $G$, where each vertex is picked with probability $p$. Then we have:
\begin{align*}
\E[e(G_p] &= p^2e(G) \\
\E[v(G_p)] &= pn \\
\E[\crn(G_p)] &= p^4\cdot \crn(G)
\end{align*}

By the above inequality, which also holds for the convex geometric subgraph $G_p$ of $G$, we get:
\[
\crn(G_p) \geq 5e(G_p) - 15 v(G_p)
\]
Taking expectations and dividing by $p^4$, yields:
\[
\crn(G) \geq \frac{5e(G)}{p^2} - 15\frac{n}{p^3}
\]
Plugging in the optimal value $p^* = \frac{9n}{2e} \le 1$ (here we used our assumption $e \ge \frac{9}{2}n$) yields the desired result.
\end{proof}

Finally, let us conclude this subsection by proving our main result concerning an upper bound on the number of edges in convex $k$-plane graphs claimed in Theorem~\ref{thm:generalupperbound}.

\begin{proof}[Proof of Theorem~\ref{thm:generalupperbound}]
Let $k \ge 5$ and let $G$ be a convex $k$-plane graph with $n$ vertices. Since $\sqrt{\frac{243}{40}k} \ge \sqrt{\frac{243}{8}} \ge 5.5>\frac{9}{2}$, the claim would be proved as soon as the number of edges of $G$ is at most $\frac{9}{2}n$, so in the following we may assume w.l.o.g. $e:=e(G)>\frac{9}{2}n$. Hence, we may apply the convex crossing lemma and find:
\[
\crn(G) \ge \frac{20}{243}\frac{e^3}{n^2}
\]
On the other hand, since $G$ is convex $k$-plane we know that every edge participates in at most $k$ crossings. By double-counting (noting that at least two edges participate in every crossing), we get that
\[
\frac{ke}{2} \ge \crn(G) \ge \frac{20}{243}\frac{e^3}{n^2}.
\]
Rearranging the inequality yields the claim. 
\end{proof}

\subsection{Proof of \Cref{prop:3planarcase} and \Cref{thm:general}}\label{app:sec:k_planar_prop1}
In this section we prove Proposition~\ref{prop:3planarcase} and \Cref{thm:general} by applying the edge-bounds established in the previous section.

\propThreePlanarcase*

\begin{proof}
For the upper bound, we show that the following algorithm always gives a valid partition into 1-planar graphs with at most $\left\lceil\frac{n}{3}\right\rceil$ colors:
W.l.o.g. assume that $P$ forms a regular $n$-gon. Now order the segments in $K(P)$ in a circular manner according to their slope, group the sets of the $n$ possible slopes into $\left\lceil\frac{n}{3}\right\rceil$ circular intervals of size $\le 3$, and make every set of segments whose slopes fall into the same interval a color class (see Figure~\ref{fig:1_planar_3_slopes}). 

\begin{figure}
\centering
\includegraphics[scale=0.4,page=1]{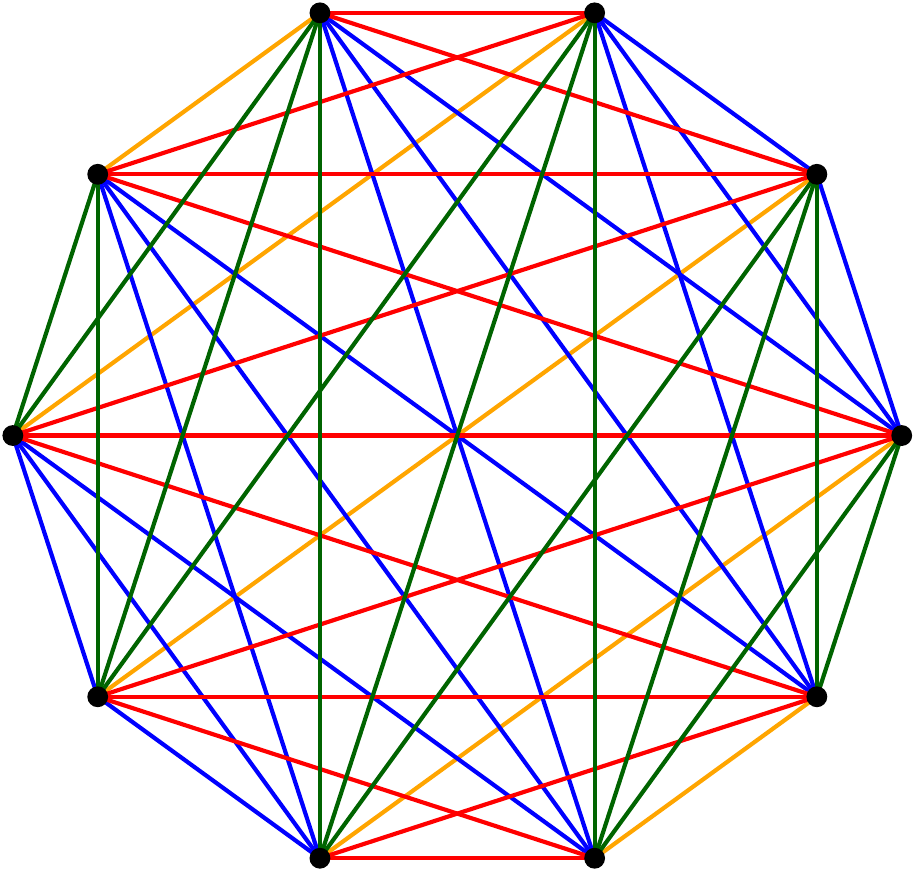}
\caption{Each color class contains all edges of up to three consecutive slope values.}
\label{fig:1_planar_3_slopes}
\end{figure}

For the lower bound, suppose we are given any $1$-planar partition $E_1,\ldots,E_c$ of $K(P)$, and let us show that $c \ge \left\lceil\frac{n}{3}\right\rceil$. Let $O \subseteq E(K(P))$ be the set of the $n$ boundary-edges of $K(P)$, and note that they are not crossed by any other edges. Hence, for every $i \in \{1,\ldots,c\}$, the geometric graph formed by the points in $P$ and the edges in $O \cup E_i$ must be $1$-planar, and hence by Proposition~\ref{prop:smallvalues} we have $|O \cup E_i| \le \frac{5}{2}n-4$ for all $i \in \{1,\ldots,c\}$. This clearly implies that $|E_i \setminus O| \le \left(\frac{5}{2}n-4\right)-n=\frac{3}{2}n-4$ for every $i$. Since the sets $E_i \setminus O, i=1,\ldots,c$ form a partition of the set of all $\frac{n(n-1)}{2}-n=\frac{n(n-3)}{2}$ interior edges of $K(P)$, we conclude:
\[
c \ge \frac{\left(\frac{n(n-3)}{2}\right)}{\frac{3}{2}n-4}=\frac{n(n-3)}{3n-8}=\frac{n}{3}-\frac{n}{3(3n-8)}>\frac{n}{3}-\frac{1}{3}.\]
(For the last inequality we used that $n \ge 5$). This implies $c \ge \lceil\frac{n}{3}\rceil$, as claimed. This concludes the proof of the proposition.
\end{proof}

\thmGeneral*

\begin{proof}
Let us first prove the upper bound. To this end, suppose that $s \ge 2$ is such that $\frac{(s-1)(s-2)}{2} \le k$, and let us show that $K(P)$ can be partitioned into $\lceil\frac{n}{s}\rceil$ $k$-planar subgraphs. W.l.o.g. assume that the points in $P$ form a regular $n$-gon. Consider all possible $n$ slopes of segments and partition them into $\lceil \frac{n}{s} \rceil$ circular intervals of size at most $s$. Then, define a color class for all edges whose slopes fall into a common interval of this partition. 

To show that all subgraphs are $\frac{(s-1)(s-2)}{2}$-planar, note that edges cannot be crossed by other edges of the same slope or clockwise slope difference $1$; by at most one edge with a clockwise slope difference $2$, by at most two edges with clockwise slope difference $3$, and so on, and symmetrically for counterclockwise slope differences. Hence, if an edge $e$ has color $i$, and if the slope of $e$ is the $j$-th slope ($j \in \{1,\ldots,s\}$) in its circular interval of slopes, then $e$ can cross with at most the following amount of edges of color $i$:
\begin{align*}
\sum_{1 \le k<j-1}{(j-k-1)}+\sum_{j+1<k \le s}{(k-j-1)}&=\frac{(j-1)(j-2)}{2}+\frac{(s-j)(s-j-1)}{2} = \\
=\frac{(s-1)(s-2)}{2}-(s-j)(j-1)
&\le \frac{(s-1)(s-2)}{2}.
\end{align*}

For the lower bound, note that $K(P)$ has $\frac{n(n-1)}{2}$ edges, and that in every $k$-planar partition of $K(P)$, every color class induces a convex $k$-plane subgraph on $n$ vertices. Hence, by Theorem~\ref{thm:generalupperbound}, every color class has size at most $\sqrt{\frac{243}{40}k} \cdot n$. So, the number of colors required in any $k$-planar partition is at least 
\[
\frac{\left(\frac{n(n-1)}{2}\right)}{\sqrt{\frac{243}{40}k} \cdot n} \ge \frac{n-1}{4.93\sqrt{k}}.
\] 
This concludes the proof. 
\end{proof}

The following intriguing question is left open by our study.

\begin{question}\label{ques:k_planar}
Is the upper bound in Theorem~\ref{thm:general} tight up to lower-order terms?
\end{question}

More generally, it would be interesting to shed some more light on the ``in-between-cases'' coming out of the upper bound in Theorem~\ref{thm:general}. 

From Theorem~\ref{thm:general} it follows directly that every convex geometric $K_n$ can be partitioned into $\lceil \frac{n}{s} \rceil$ $k$-planar subgraphs where $s$ is the smallest integer such that $\frac{(s-1)(s-2)}{2} \ge k$. However, it is very natural to ask whether in the case that $k$ is far from this upper bound $\frac{(s-1)(s-2)}{2}$, we may be able to improve upon $\lceil \frac{n}{s} \rceil$, or whether we can show a matching lower bound. This question is surprisingly difficult. For example, for $k=2$ we can show a lower bound of $\frac{3n}{10}$, almost matching the upper bound $\frac{n}{3}$ (using computer assistance); we, however, omit this result from this version.

\section{Partitions into $\mathbf{k}$-\quasiplanar subgraphs and spanning trees}\label{sec:k-quasi-planar}

\newcommand\cquasi{\ensuremath{\lceil\frac{n}{k-1}\rceil}}
\newcommand\cquasin{\ensuremath{\lceil\frac{n}{2(k-1)}\rceil}}

In this section, we develop bounds on the number of colors required in a $k$-\quasiplanar partition for point sets in general position (for $k=2$ this again amounts to the setting of plane subgraphs, hence we assume $k \geq 3$ in the following).
The setting of spanning trees is resolved by \Cref{lem:3_quasiplanar_st}:

\lemThreeQuasiplanarST*

\begin{proof}
Order the points in $P$ by increasing $x$-coordinate (w.l.o.g. no two points have the same $x$-coordinate, otherwise pick a different direction) and label them as $p_1, \dots, p_{2n}$.
Further, we will distinguish the points into evenly indexed points $p_{2i}$ for $i = 1, \dots, n$ and oddly indexed points $p_{2i-1}$ for $i = 1, \dots, n$.

The goal is to define $n$ double stars.
We construct a double star $T_i$ for any two consecutive vertices $\{p_{2i-1}, p_{2i}\}$, for $i = 1, \dots, n$.
To construct $T_i$, we connect $p_{2i-1}$ with all evenly indexed points left of $p_{2i-1}$, that is, $p_{2j}$ for $j = 1,\dots,i-1$ and all oddly indexed points right of $p_{2i-1}$, that is, $p_{2j-1}$ for $j = i+1, \dots, n.$
Further, we connect $p_{2i}$ with the remaining vertices, namely with all oddly indexed points left of $p_{2i}$, that is, $p_{2j-1}$ for $j=1,\dots,i$ and all evenly indexed points right of $p_{2i}$, that is, $p_{2j}$ for $j = i+1, \dots, n$.
Note that this includes the edge $p_{2i-1}p_{2i}$.
$T_i$ indeed forms a double star, as all vertices are connected to either $p_{2j-1}$ or $p_{2j}$. Thus, it also covers all vertices and forms a spanning tree.
It remains to show, that every edge is covered exactly once.
By construction, any edge between two vertices of the same parity, that is, $p_{2i}$ and $p_{2j}$ or $p_{2i-1}$ and $p_{2j-1}$ for some $1 \leq i < j \leq n$ belongs to the spanning tree $T_i$ of the left vertex.
Conversely, for an edge between two vertices of different parity, that is, $p_{2i-1}$ and $p_{2j}$ or $p_{2i}$ and $p_{2j-1}$ for some $1 \leq i \leq j \leq n$, the edge belongs to the spanning tree $T_j$ of the right vertex.
As these two cases cover all edges, $T_1, \dots, T_n$ is indeed a partition.
Noting that any double star is necessarily 3-\quasiplanar, the claim follows.
\end{proof}

So, we turn our attention to the subgraph setting.
Consider a point set $P$ of size $2n$ with a crossing family of size $n$. Let $\ell_1, \dots, \ell_n$ be the corresponding halving lines. We label these lines in clockwise order, more precisely their intersections with a sufficiently large circle appear in clockwise order.
For each halving line $\ell_i$ define an infinitesimally counter-clockwise rotated line $\ell'_i$, such that the two defining vertices (say $p_i,q_i$) of this line lie to either side of $\ell'_i$. Define $\ell_i^+$ to be the upper halfplane (bounded by $\ell'_i$) and let it contain $p_i$; similarly $\ell_i^-$ denotes the lower halfplane (bounded by $\ell'_i$) and it contains $q_i$. See \autoref{fig:quasiplanar_definitions} for an illustration.

  \begin{figure}
 	\centering
 	\includegraphics[scale=0.5, page=1]{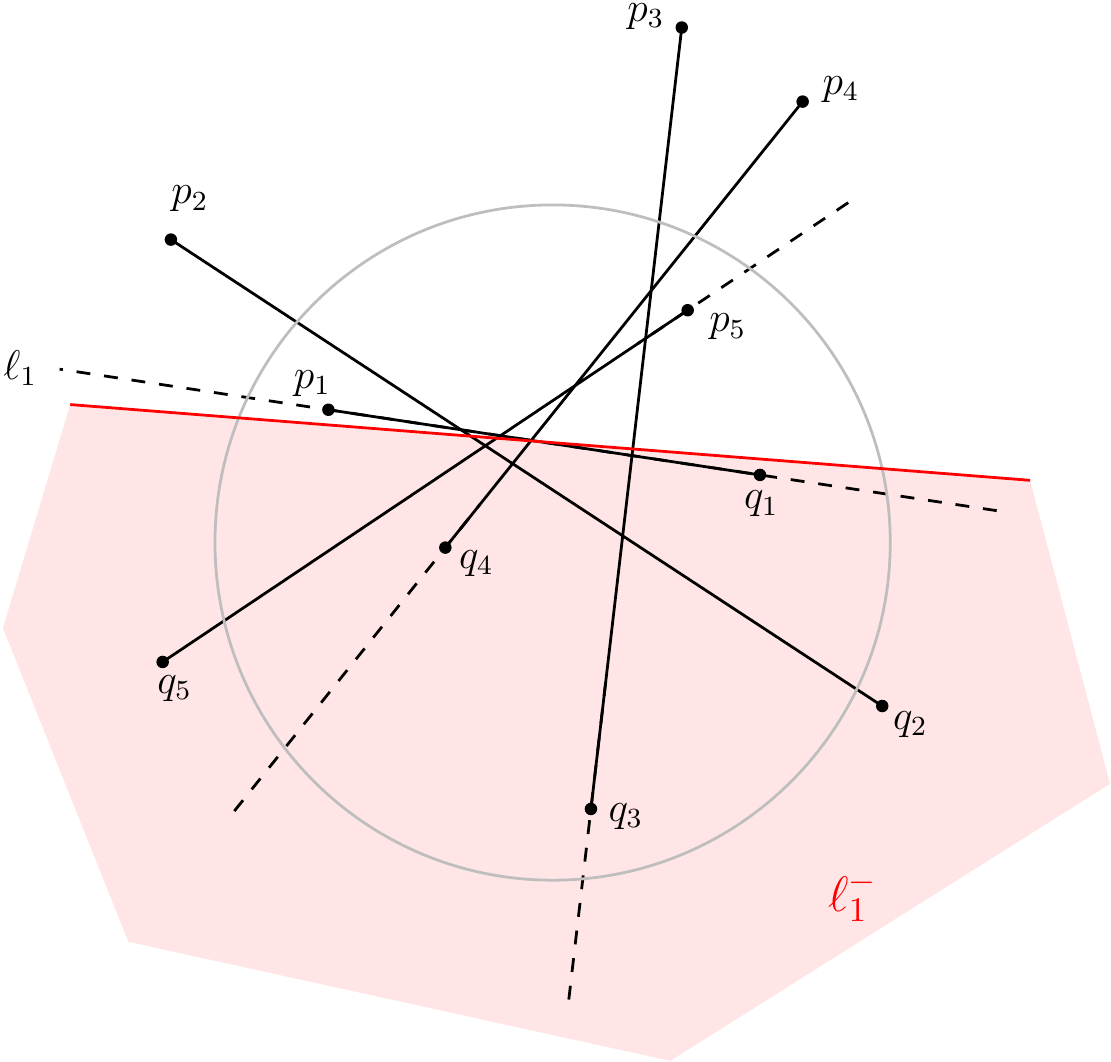}
 	\caption{The labeling of the halving lines $\ell_1, \dots, \ell_n$ in clockwise order. Illustration of the definition of the halfplane $\ell_1^-$. Note that $p_i \in \ell_1^+$ and $q_i \in \ell_1^-$ for any $i$.}
 	\label{fig:quasiplanar_definitions}
 \end{figure}
 
\begin{restatable}{lemma}{lemHalvingquasi}\label{lem:halvingquasi}
	Let $P$ be a point set of size $2n$, with a crossing family of size $n$, then $\cquasi$ colors are required and sufficient to partition $K(P)$ into $k$-\quasiplanar subgraphs.
\end{restatable}

\begin{proof}
	Since we have a crossing family of size $n$, we need at least $\cquasi$ different colors to partition $K(P)$ into $k$-\quasiplanar subgraphs.
	
	The other direction is a bit more involved. We build $c := \cquasi$ subgraphs $G_1, \ldots, G_{c}$ of $K(P)$ as follows.
	Each subgraph $G_i$ in turn is formed by the union of three subgraphs.

To construct $G_1$, let $X_1$ be the collection of defining vertices of the first $k-1$ consecutive halving lines starting from $\ell_1$, that is, $X_1 := \{p_1, \ldots, p_{k-1}, q_1, \ldots, q_{k-1}\}$.
Next, we consider all points in $\ell_1^+$ and form the complete bipartite graph $B_1^+$ between points both in $\ell_1^+$ and in $X_1$, that is, point set $X_1 \cap \ell_1^+$ and points in $\ell_1^+$ but not in $X_1$, that is, point set $(P \setminus X_1) \cap \ell_1^+$.
Symmetrically, we form the complete bipartite graph $B_1^-$ between the point set $X_1 \cap \ell_1^-$ and the point set $(P \setminus X_1) \cap \ell_1^-$.
An illustration of the construction is given in \Cref{fig:k-quasi-bipartite}.
The subgraph $G_1$ is finally defined to be the union of $K(X_1), B_1^+$ and $B_1^-$.

	We iteratively repeat the same process for the next $k-1$ halving lines until reaching~$p_n$. More precisely, $G_l$ consists of the union of the complete graph with vertex set
	$X_l = \{p_i,q_i\,|\,(l-1)(k-1)<i \le \min(l(k-1),n)\}$ 
	and the two bipartite graphs defined by $\ell_{(l-1)(k-1)+1}$ as before.
The last graph $G_{c}$ may be formed by less than $k-1$ halving lines.
	
	\begin{figure}
		\centering
		\includegraphics[scale=0.45, page=1]{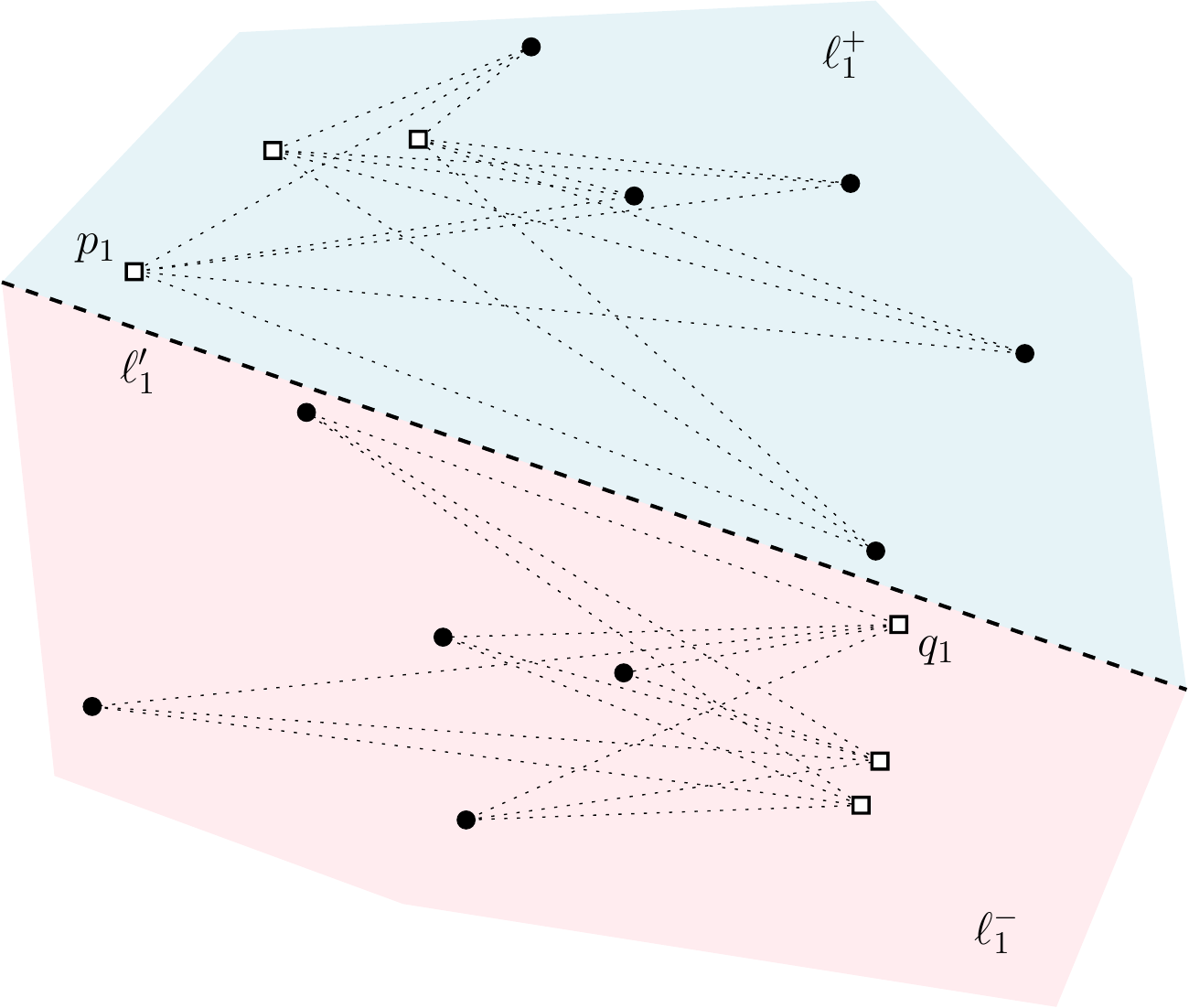}
		\caption{The vertices of $X_1$ are represented by squares. The dotted edges in the blue (red) region represent the complete bipartite subgraph $B_1^+$ ($B_1^-$) of $G_1$ corresponding to the line $\ell_1$ with the defining vertices $p_1$ and $q_1$.}
			\label{fig:k-quasi-bipartite}
	\end{figure}
	
	We first validate that each $G_i$ is $k$-\quasiplanar.
	For notational convenience we prove it only for $G_1$ in the following, though the same argument works for any $G_i$.
	The two bipartite subgraphs $B_1^+$ and $B_1^-$ are disjoint, as they lie to either side of the line $\ell'_1$.
	Thus, any crossing family contains only edges from at most one of them, and potentially further edges from the complete subgraph $K(X_1)$.
	Therefore every edge from the crossing family must be incident to one of the vertices that are both in $X_1$ and on the respective side of $\ell_1$.
	By construction, there can be at most $k-1$ such vertices and therefore, any crossing family has at most size $k-1$ as well.  
	
	It remains to show that any edge of $K(P)$ is covered by some $G_i$. Let $e := \{u, v\}$ be an edge of $K(P)$.
	For every $i = 1,\dots, n$, $p_i$ and $q_i$ are part of $X_{\lceil i/(k-1) \rceil}$.
	Thus, the endpoint $u$ is contained in $X_r$ for some $r$, and $v$ is contained in $X_s$, for some $s$.
	If $r = s$ we are done, as then $e$ is contained in $K(X_r)$ and thus part of $G_r$.
	So suppose $r \neq s$.
	Then $v$ and $u$ lie on the same side of either $\ell'_r$ or $\ell'_s$.
	In the former case, $e$ would be contained in either of the two complete bipartite subgraphs of $G_r$, that is, $B_r^+$ or $B_r^-$. In the latter, $e$ is contained in~$G_s$.
	\end{proof}

\thmMainQuasiplanarCF*

\begin{proof}
Let $P' \subseteq P$ be the subset of endpoints induced by a largest crossing family of size $m$.

Then, the lower bound follows immediately from \Cref{lem:halvingquasi} applied on $P'$.

For the upper bound, divide the point set $P \setminus P'$ into disjoint subsets $Q_1, \dots, Q_c$ of size $k-1$, where $c = \lceil \frac{n-2m}{k-1} \rceil$. Each edge with an endpoint in some $Q_i$ assign the color $i$ (for edges that have two choices, pick one arbitrarily). Certainly, each color class is $k$-\quasiplanar, since it consists of (at most) the union of $k-1$ stars. Together with $K(P')$, which we can clearly partition by using $\lceil \frac{m}{k-1} \rceil$ colors, the upper bound follows.
\end{proof}

\section{Conclusion}

In this paper, we initiated the study of partitions into beyond planar subgraphs. We studied $k$-planar partitions for the case of convex position and showed a lower bound of $\frac{n-1}{4.93\sqrt{k}}$ and an upper bound of $\frac{n}{\sqrt{2k}}$. Moreover, for 1-planar partitions we proved the tight bound $\lceil \frac{n}{3} \rceil$.

The $k$-\quasiplanar setting turned out to be more accessible, where we showed that any complete geometric graph can be partitioned into 3-\quasiplanar spanning trees. The corresponding question concerning $k$-planar partitions remains open:

\begin{question}
Does there exist a constant $k$ such that any complete geometric graph $K(P)$ on $2n$ vertices can be partitioned into $n$ $k$-planar spanning trees?
\end{question}

Moreover, we showed upper and lower bounds for $k$-\quasiplanar partitions for general $k$ (depending on the size of a largest crossing family).



\bibliography{../../partition_literature}

\appendix

\end{document}